\documentclass{amsart}
\usepackage{amsfonts,amssymb,amscd,amsmath,enumerate,verbatim,calc}
\usepackage[all]{xy}

\newcommand{\wrt}{with respect to}

\newcommand{\m}{\mathfrak{m} }
\newcommand{\M}{\mathcal{M} }
\newcommand{\q}{\mathfrak{q} }

\newcommand{\R}{\mathcal{R} }
\newcommand{\Sc}{\mathcal{S} }

\newcommand{\Z}{\mathbb{Z} }

\newcommand{\rt}{\rightarrow}

\newcommand{\wh}{\widehat }
\newcommand{\wt}{\widetilde }

\newcommand{\charp}{\operatorname{char}}
\newcommand{\Proj}{\operatorname{Proj}}
\newcommand{\Spec}{\operatorname{Spec}}

\newcommand{\coker}{\operatorname{coker}}

\theoremstyle{plain}

\newtheorem{theorem}{Theorem}[section]
\newtheorem{corollary}[theorem]{Corollary}

\newtheorem{proposition}[theorem]{Proposition}

\theoremstyle{definition}

\newtheorem{s}[theorem]{}
\newtheorem{remark}[theorem]{Remark}

\newtheorem{construction}[theorem]{Construction}
\theoremstyle{remark}

\begin{document}

\title[Cohomology Vanishing Theorems]{Cohomology Vanishing theorems over some rings containing nilpotents}
\author{Tony~J.~Puthenpurakal}
\date{\today}
\address{Department of Mathematics, IIT Bombay, Powai, Mumbai 400 076, India}

\email{tputhen@gmail.com}
\subjclass{Primary 13D45, 13A30; Secondary 14B15}
\keywords{local cohomology, graded local cohomology, associated graded rings, Rees algebras, holonomic modules, $F$-finite $F$-modules}

 \begin{abstract}
In this paper we study  some un-expected vanishing theorems of local cohomology modules of  certain graded rings with nilpotents.

(1) Let $(A,\m)$ be complete Noetherian local ring of dimension $d$ and let $P$ be a prime ideal with $G_P(A) = \bigoplus_{n \geq 0}P^n/P^{n+1}$ a domain.  Fix $r \geq 1$. If $J$ is a homogeneous ideal of $G_{P^r}(A)$ with $\dim G_{P^r}(A)/J > 0$ then the local cohomology module $H^d_J(G_{P^r}(A)) = 0$.

(2) Let $A = K[[X_1, \ldots,X_d]]$ and let $\m = (X_1, \ldots, X_d)$. Assume $K$ is separably closed. Fix $r \geq 1$. Let $J$ be a homogeneous ideal of $G_{\m^r}(A)$.
We show that local cohomology modules $H^{j}_J(G_{\m^r}(A)) = 0$ for $j \geq d -1$ if and only if
  $\dim G_{\m^r}(A)/J \geq 2$ and $\Proj G_{\m^r}(A)/J $ is connected.
\end{abstract}
 \maketitle
\section{introduction}
Local cohomology is a fundamental tool in algebraic geometry and commutative algebra. So vanishing results are very useful. Let $(A,\m)$ be either a Noetherian local ring or a standard graded Noetherian ring with $\m$ its maximal homogeneous ideal with $\dim A = d$. Let $I \subseteq \m$ be an ideal (it is homogeneous if $A$ is graded). The first vanishing result is that $H^i_I(A) = 0$ for all $i > d$; see \cite[6.1.2]{BS}. Perhaps the next result is the famous Hartshorne-Lichtenbaum theorem, which states that if $A$ is a complete local domain (in the graded case $A_0$ is a complete domain and $A$ is a graded domain)  then $H^d_I(A) = 0$ if and only if $\dim A/I  > 0$; see \cite[14.1]{I24} (for the graded case see \cite[14.1.16]{BS}). There is a version of this result which is applicable to all Noetherian local rings but it is complicated to state; see \cite[14.6]{I24}. The next case is when $A$ is a complete regular local ring containing a separably closed residue field then $H^{j}_I(A)  = 0$ for $j \geq d -1$ if and only id $\dim A/I \geq 2$ and
$\Spec(A) \setminus \{ \m \}$ is connected, see \cite[14.7]{I24}. Analogous result holds in the graded case when $A = k[X_1,\ldots, X_d]$, see \cite[7.5]{H}.

Usually local cohomology behaves well when the underlying rings are nice. In this paper we prove vanishing theorems for certain rings containing nilpotents. We stress that these results are completely unexpected.

We first prove
\begin{theorem}\label{m-1}
Let $(A,\m)$ be complete Noetherian local ring of dimension $d$ and let $P$ be a prime ideal with $G_P(A) = \bigoplus_{n \geq 0}P^n/P^{n+1}$ a domain.  Fix $r \geq 1$. If $J$ is a homogeneous ideal of $G_{P^r}(A)$ with $\dim G_{P^r}(A)/J > 0$ then the local cohomology module $H^d_J(G_{P^r}(A)) = 0$.
\end{theorem}

Next we show
\begin{theorem}
\label{m-2} Let $(A,\m)$ be a regular local ring with $k = A/\m$ separably closed. Fix $r \geq 1$. Let $J$ be a homogeneous ideal of $G_{\m^r}(A)$.
If $\dim G_{\m^r}(A)/J \geq 2$ and $\Proj G_{\m^r}(A)/J $ is  connected then
the local cohomology modules $H^{j}_J(G_{\m^r}(A)) = 0$ for $j \geq d -1$.
\end{theorem}
\begin{remark}
Note that in Theorem \ref{m-2} we \emph{do not} assume $A$ is equi-characteristic. Also note that for $r \geq 2$ the ring  $G_{\m^r}(A)$ may not contain a field.
\end{remark}
It is perhaps worthwhile to investigate converse of Theorem \ref{m-2}. We prove:
\begin{theorem}
\label{m-2c} Let $(A,\m)$ be an equi-characteristic regular local ring with $k = A/\m$ separably closed. Fix $r \geq 1$. Let $J$ be a homogeneous ideal of $G_{\m^r}(A)$.
If the local cohomology modules $H^{j}_J(G_{\m^r}(A)) = 0$ for $j \geq d -1$,
then $\dim G_{\m^r}(A)/J \geq 2$ and $\Proj G_{\m^r}(A)/J $ is  connected.
\end{theorem}

We now describe in brief the contents of this paper. In section two we discuss a construction that we need. In section three we discuss an infinitely generated $\R =A[It]$-module $W^I(A)$ that plays a crucial role in our proofs. In section four we prove Theorem \ref{m-1}. In section five we prove Theorem \ref{m-2}. {In the next section we discuss some preliminary results that is needed to prove
Theorem \ref{m-2c}. In section seven we discuss $F$-module or $D$-module structure on $W^\m(A)$. In section eight we give proof of Theorem \ref{m-2c}.

\section{A construction}
In this section $A$ is a Noetherian local ring, $I$ an ideal of $A$. Let $\R(I) = A[It]$ be the Rees algebra of $I$.  Let $G_I(A) = \bigoplus_{n \geq 0}I^n/I^{n+1}$ be the associated graded ring of $I$. Fix $r \geq 1$. In this section we discuss a construction of to relate homogeneous ideals in $G_{I^r}(A)$ with a certain ideal in $G_I(A)$.
The strategy is to use the Rees algebra. We note that ideals in $G_I(A)$ correspond to ideals of $\R(I)$ containing $It^0$.
Throughout $\ ^{<r>}$ denotes the $r^{th}$-Veronese functor. We note that the associated graded ring does not behave well \wrt \ the Veronese functor. However the Rees algebra does.
We have $\R(I)^{<r>} = \R(I^r)$.
\begin{construction}\label{const}
  Let $Q$ be a homogeneous ideal in $G_{I^r}(A)$. Let $\R(I^r) = A[I^r u]$. Let $\eta \colon \R(I^r) \rt G_{I^r}(A)$ be the natural map.  Let $K = \eta^{-1}(Q)$.
  Say $K = (\q u^0, k_1u^{a_1}, k_2u^{a_2}, \cdots, k_su^{a_s})$. Here $k_i \in I^{a_ir}$  (with $a_i \geq 1$) and $\q$ is an ideal of $A$ containing $I^r$.
  Let $a = lcm \{ a_i \}_{1\leq i \leq s}$. Consider
  the following ideal of $\R(I^r)$
  \[
  \wt{K} = (\q u^0, c_1u^a, c_2u^a, \cdots, c_su^a) \quad \text{where} \ c_i = k_i^{a/a_i} \ \text{for} \ i = 1, \ldots,s.
  \]

  Consider the following ideal of $\R(I) = A[It]$;
  \[
  K^\sharp = (\q + I, c_1t^{ar}, c_2t^{ar}, \cdots, a_st^{a_sr}).
  \]
\end{construction}

Our construction has the following properties:
\begin{proposition}\label{const-prop}(with hypotheses as in \ref{const}). We have
  \begin{enumerate}[\rm (1)]
  \item $\sqrt{\wt{K}} = \sqrt{K}$
  \item $(K^\sharp)^{<r>} =  (\q + I, c_1u^a, c_2u^a, \cdots, c_su^a) $.
  \item $ \sqrt{(K^\sharp)^{<r>}} = \sqrt{\wt{K}} = \sqrt{K}$.
  \item $\dim A[It]/ K^\sharp = \dim A[I^ru]/K$.
  \item $\Proj A[It]/ K^\sharp$ is connected if and only if $\Proj A[I^ru]/K$ is connected.
  \item $K^\sharp$ contains $It^0$. Let $\xi \colon A[It] \rt G_I(A)$ be the natural map. Then
  $A[It]/ K^\sharp = G_I(A)/\xi(K^\sharp)$.
\end{enumerate}
\end{proposition}
\begin{proof}
(1) This is clear.

(2) Set $J = (\q + I, c_1u^a, c_2u^a, \cdots, c_su^a) $. Clearly $J \subseteq (K^\sharp)^{<r>} $.
Let $\xi \in (K^\sharp)^{<r>}_n = K^\sharp_{nr}$. We have
$$ \xi = \sum_{i =1}^{m}\alpha_i w_it^0 + \sum_{j = 1}^{l}\beta_jc_jt^{ar},$$
where $w \in q + I$ and $\beta_j \in I^{nr - ar}$. Note $\beta_j \in I^{r(n-a)}$. It follows that
$$ \xi = \sum_{i =1}^{m}\alpha_i w_iu^0 + \sum_{j = 1}^{l}\beta_jc_ju^a  \in J.$$
The result follows.

(3) We note that $\sqrt{\q + I} = \sqrt{\q} $ as $\q \supseteq I^r$. It follows that $ \sqrt{(K^\sharp)^{<r>}} = \sqrt{\wt{K}}$. The second equality follows from (1).

(4), (5).  Let $S =  A[It]/ K^\sharp$. Them $S^{<r>} = A[I^ru]/(K^\sharp)^{<r>}$.  We note that $S^{<r>} \subseteq S$ is a finite integral extension. It follows that
\[
\dim  A[It]/ K^\sharp = \dim A[I^ru]/(K^\sharp)^{<r>} = \dim A[I^ru]/K.
\]
Here the second equality follows from (3).

We also have $\Proj S \cong \Proj S^{<r>}$. We also have that by (3) we get as topological spaces (but not as schemes)
$$ \Proj A[I^ru]/(K^\sharp)^{<r>} = \Proj A[I^ru]/K.$$
The result follows.

(6) This is clear.
\end{proof}

\section{A module over the Rees algebra}
In this section we recall construction of a non-finitely generated module over the Rees algebra  which has proved useful in study of associated graded rings.
\s Let  $(A, \m)$ be Noetherian local ring and let $I$ be an ideal of $A$. Let $\R(I) = A[It]$ be the Rees algebra of $I$. Note $\R$ is a subring of $A[t]$.
Set $W^I(A) = \bigoplus_{n \geq 1}A/I^n = A[t]/\R(I)$. Then $W^I(A)$ is an $\R(I)$-module. Note it is \emph{not}  finitely generated as an $\R(I)$-module.
Let $\Sc(I) = A[It, t^{-1}]$ be the extended Rees algebra of $I$. Note it is a subring of $A[t, t^{-1}]$ and $W^I(A) = A[t, t^{-1}]/\Sc(I)$ is also an $\Sc(I)$-module. Furthermore the $\R(I)$-module structure on $W^I(A)$ is through the restriction of scalars along the inclusion $\R(I) \rt \Sc(I)$

\begin{remark}
In \cite{Pu5} we defined $L^I(A) = \bigoplus_{n \geq 0}A/I^{n+1} = W^I(A)(1)$. However in this paper we deal with $W^I(A)$. For the reason of this see, \ref{rachel}.
\end{remark}

\s \label{first}
We have exact sequence $$ 0 \rt I^n/I^{n+1}  \rt A/I^{n+1} \rt A/I^n \rt 0$$
for all $n \geq 0$. So we have an exact sequence of $\Sc(I)$-modules (and hence of $\R(I)$-modules).
\begin{equation*}
  0 \rt G_I(A) \rt W^I(A)(1) \xrightarrow{t^{-1}} W^I(A) \rt 0. \tag{$\dagger$}
\end{equation*}
\s $W^I(A)$ behaves well \wrt \ the Veronese functor. We have an equality of $\R(I^r) = \R(I)^{<r>}$-modules
$$ (W^I(A))^{<r>} = \bigoplus_{n \geq 1}A/I^{nr}  = W^{I^r}(A). $$
\section{Proof of Theorem \ref{m-1}}
In this section we give
\begin{proof}[Proof of Theorem \ref{m-1}]
Let  $J$ be an ideal of $G_{P^r}(A)$ with $\dim G_{P^r}(A)/J > 0$. Let $\eta \colon A[P^ru] \rt G_{P^r}(A)$ be the natural map. Let $K = \eta^{-1}(J)$. We make the construction as in \ref{const}.
By \ref{const-prop} $\dim A[Pt]/K^\sharp > 0$. Let  $\xi \colon A[Pt] \rt G_P(A)$ be the natural map. Set $I = \xi(K^\sharp)$. Then by \ref{const-prop};
  $A[Pt]/ K^\sharp = G_P(A)/I$. So we have $H^d_{I}(G_P(A))  = 0$. The exact sequence $(\dagger)$ induces exact sequence in cohomology
  $$\cdots  \rt H^i_I(G_P(A)) \rt H^i_{K^\sharp}(W^P(A))(1) \xrightarrow{t^{-1}} H^i_{K^\sharp}(W^P(A)) \rt H^{i+1}_I(G_P(A)) \rt \cdots. $$
As $H^d_I(G_P(A)) = 0$ we get the maps
\begin{enumerate}
  \item $H^{d-1}_{K^{\sharp}}(W^P(A))_{n+1} \xrightarrow{t^{-1}} H^{d-1}_{K^{\sharp}}(W^P(A))_{n}$ is surjective for all $n \in \Z$; and
  \item $H^{d}_{K^{\sharp}}(W^P(A))_{n+1}  \xrightarrow{t^{-1}} H^{d}_{K^{\sharp}}(W^P(A))_{n}$ is injective for all $n \in \Z$.
\end{enumerate}
As  local cohomology behaves well \wrt \ the Veronese functor we have
for all $i \geq 0$
$$H^i_{K^\sharp}(W^{P}_A)^{<r>}  = H^i_{(K^\sharp)^{<r>}}(W^{P^r}_A).$$
So we have
$$H^i_{(K^\sharp)^{<r>}}(W^{P^r}_A)_n = H^i_{K^\sharp}(W^{P}_A)_{nr}. $$
Also note that
$$H^i_{(K^\sharp)^{<r>}}(W^{P^r}_A)(1)_n = H^i_{(K^\sharp)^{<r>}}(W^{P^r}_A)_{n+1} = H^i_{K^\sharp}(W^{P}_A)_{(n+1)r}. $$
We have a surjective map
$$ H^{d-1}_{K^{\sharp}}(W^P(A))_{nr + r} \xrightarrow{t^{-r}} H^{d-1}_{K^{\sharp}}(W^P(A))_{nr},  \quad \text{for all}  \ n \in \Z. $$
Notice that multiplication by $t^{-r}$ is same as multiplication by $u^{-1}$. So we have surjective maps
$$ H^{d-1}_{(K^\sharp)^{<r>}}(W^{P^r}(A))_{n + 1} \xrightarrow{u^{-1}} H^{d-1}_{(K^\sharp)^{<r>}}(W^{P^r}(A))_{n},  \quad \text{for all}  \ n \in \Z. $$
Similarly we have injective maps
$$ H^{d}_{(K^\sharp)^{<r>}}(W^{P^r}(A))_{n + 1} \xrightarrow{u^{-1}} H^{d}_{(K^\sharp)^{<r>}}(W^{P^r}(A))_{n},  \quad \text{for all}  \ n \in \Z. $$
We have a short exact sequence of $A[P^ru]$-modules
$$ 0 \rt G_{P^r}(A) \rt W^{P^r}(A)(1) \xrightarrow{u^{-1}} W^{P^r}(A) \rt 0.$$
Taking local cohomology \wrt $(K^\sharp)^{<r>}$ and the above discussion we have
$$H^d_{(K^\sharp)^{<r>}}(G_{P^r}(A)) = 0.$$
By \ref{const-prop}, we get that
$$H^i_J(G_{P^r}(A))  \cong H^i_{(K^\sharp)^{<r>}}(G_{P^r}(A)) \quad \text{for all} \ i \geq 0.$$
The result follows.
\end{proof}
\section{Proof of Theorem \ref{m-2}}
In this section we give
\begin{proof}[Proof of Theorem \ref{m-2}]
By \ref{m-1} we get that $H^{d}_J(G_{\m^r}(A)) = 0$. Let $\eta \colon A[\m^ru] \rt G_{\m^r}(A)$ be the natural map. Let $K = \eta^{-1}(J)$. We make the construction as in \ref{const}.
By \ref{const-prop} $\dim A[\m t]/K^\sharp \geq 2$ and $\Proj A[\m t]/K^\sharp$ is connected. Let  $\xi \colon A[\m t] \rt G_\m(A)$ be the natural map. Set $I = \xi(K^\sharp)$. Then by \ref{const-prop};
  $A[\m t]/ K^\sharp = G_\m(A)/I$. So we have $H^j_{I}(G_\m(A))  = 0$  for $j \geq d - 1$. The exact sequence $(\dagger)$ induces exact sequence in cohomology
  $$\cdots  \rt H^i_I(G_\m(A)) \rt H^i_{K^\sharp}(W^\m(A))(1) \xrightarrow{t^{-1}} H^i_{K^\sharp}(W^\m(A)) \rt H^{i+1}_I(G_\m(A)) \rt \cdots. $$
As $H^{d-1}_I(G_P(A)) = 0$ we get the maps
\begin{enumerate}
  \item $H^{d-2}_{K^{\sharp}}(W^\m(A))_{n+1} \xrightarrow{t^{-1}} H^{d-2}_{K^{\sharp}}(W^\m(A))_{n}$ is surjective for all $n \in \Z$; and
  \item $H^{d-1}_{K^{\sharp}}(W^\m(A))_{n+1}  \xrightarrow{t^{-1}} H^{d-1}_{K^{\sharp}}(W^\m(A))_{n}$ is injective for all $n \in \Z$.
\end{enumerate}
By an argument similar to proof of Theorem \ref{m-1} we have
 surjective maps
$$ H^{d-2}_{(K^\sharp)^{<r>}}(W^{\m^r}(A))_{n + 1} \xrightarrow{u^{-1}} H^{d-2}_{(K^\sharp)^{<r>}}(W^{\m^r}(A))_{n},  \quad \text{for all}  \ n \in \Z. $$
Similarly we have injective maps
$$ H^{d-1}_{(K^\sharp)^{<r>}}(W^{\m^r}(A))_{n + 1} \xrightarrow{u^{-1}} H^{d-1}_{(K^\sharp)^{<r>}}(W^{\m^r}(A))_{n},  \quad \text{for all}  \ n \in \Z. $$
We have a short exact sequence of $A[\m^ru]$-modules
$$ 0 \rt G_{\m^r}(A) \rt W^{\m^r}(A)(1) \xrightarrow{u^{-1}} W^{\m^r}(A) \rt 0.$$
Taking local cohomology \wrt $(K^\sharp)^{<r>}$ and the above discussion we have
$$H^{d-1}_{(K^\sharp)^{<r>}}(G_{\m^r}(A)) = 0.$$
By \ref{const-prop} it follows that
$$H^i_J(G_{\m^r}(A))  \cong H^i_{(K^\sharp)^{<r>}}(G_{\m^r}(A)) \quad \text{for all} \ i \geq 0.$$
The result follows.
\end{proof}
\section{Some preliminaries to prove Theorem \ref{m-2c}}
We need a few preliminaries to prove Theorem \ref{m-2c}.
\s Assume $\charp K = 0$. Consider the Weyl algebra
$$A_m(K) = K<X_1,\ldots, X_m, \partial_1, \ldots, \partial_m>.$$
 We give the grading on $A_m(K)$ as $\deg X_i = 1$ and $\deg \partial_i = -1$.   Let $\mathcal{E} = \sum_{i = 1}^{m} X_i \partial_i$ be the Eulerian operator. We say a graded $A_m(K)$-module is  generalized Eulerian if for a homogeneous $m \in M$ there exists  integer $a$ (depending on $m$) such that $(\mathcal{E} - \deg m)^am = 0$.

 \s Assume $\charp K = p > 0$. Then we have the Frobenius map $F \colon R \rt R$ given by
 $F(r) = r^p$. There is a notion of $F$ and $F$-finite $R$-modules, see \cite{Lyu-2}.

\begin{theorem}\label{tame-body}[see \cite[5,1, 6.1]{P}, \cite[7.1, 7.2]{P2}]
 Let $R = K[X_1,\ldots, X_m]$ be standard graded with
  $K$ an infinite field.
 Let $\M = \bigoplus_{n \in \Z} \M_n$ be a graded $R$-module. If $\charp K  = p > 0$ assume  $\M$ is a graded $F_R$-finite $F_R$-module. If $\charp K = 0$ assume $\M$ is a graded generalized Eulerian, holonomic $A_m(K)$-module.
   Then we have
\begin{enumerate}[\rm (a)]
\item
If $\M_n = 0$ for $|n| \gg 0$ then $\M = 0$.
\item
The following assertions are equivalent:
\begin{enumerate}[\rm(i)]
\item
$\M_n \neq 0$ for infinitely many $n \ll 0$.
\item
There exists $r$  such that $\M_n \neq 0$ for all $n \leq r$.
\item
$\M_n \neq 0$ for all $n \leq -m$.
\item
$\M_n \neq 0$ for some $n \leq -m$.
\end{enumerate}
\item
The following assertions are equivalent:
\begin{enumerate}[\rm(i)]
\item
$\M_n \neq 0$ for infinitely many $n \gg 0$.
\item
There exists $r$  such that $\M_n \neq 0$ for all $n \geq r$.
\item
$\M_n \neq 0$ for all $n \geq 0$.
\item
$\M_n \neq 0$ for some $n \geq 0$.
\end{enumerate}
\end{enumerate}
\end{theorem}
As a consequence we get:
\begin{corollary}\label{cor-inf-vanish}
(with hypotheses as in \ref{tame-body}) If $\M_n = 0$ for infinitely many $n \gg 0$ and infinitely many $n \ll 0$ then $\M = 0$.
\end{corollary}
\section{$D$-module or $F_R$-module structure on $W$}
In this section we assume $A = k[[X_1, \ldots, X_d]], \m = (X_1, \ldots, X_d)$.   Let $G = G_\m(A)$ be the associated graded ring of $A$ \wrt \ $\m$. We note that
$G \cong k[X_1, \ldots, X_d]$.
Let $\R = A[\m t]$ be the Rees algebra of $A$ \wrt \ $\m$ and let $\Sc = A[\m t,t^{-1}]$ be the extended Rees algebra of $A$ \wrt \ $\m$. Let $I$ be a homogeneous  ideal in $\R$.
The exact sequence of $\Sc$-modules (and hence of $\R$-modules)
$$ 0 \rt G \rt W^\m(A)(1)  \xrightarrow{t^{-1}} W^\m(A) \rt 0$$
induces exact sequence in cohomology
$$ \cdots \rt H^i_I(G) \rt H^i_I(W^\m(A))(1) \xrightarrow{t^{-1}} H^i_I(W^\m(A)) \rt H^{i+1}_I(G) \rt \cdots.$$
Let $$U_i =\ker (H^i_I(W^\m(A))(1) \xrightarrow{t^{-1}} H^i_I(W^\m(A))),  \text{and} $$
$$ V_i =\coker (H^i_I(W^\m(A))(1) \xrightarrow{t^{-1}} H^i_I(W^\m(A))). $$
For all $i \geq 0$, we have short-exact sequences of $G$-modules
$$ 0 \rt V_{i-1} \rt H^i_I(G) \rt U_i \rt 0.$$
The goal of this section is to study properties of the above exact sequence

 \s Let $\Sc = A[\m t, t^{-1}]$ be the extended Rees algebra of $A$ \wrt \ $\m$. We note that $\Sc$ is a regular ring. Recall $W^\m(A)$ is also a $\Sc$-module
 and  the $\R$-module structure of $W^\m(A)$ is through the restriction of scalars along the inclusion
 $\R \rt \Sc$. We also note that $A[t, t^{-1}] = \Sc_{t^{-1}}$.

 \s\label{char-p} We first consider the case when $\charp k = p > 0$. We note that $\Sc$ is $F_\Sc$-finite $F_\Sc$-module. It follows that $A[t, t^{-1}] = \Sc_{t^{-1}}$ is $F_\Sc$-finite $F_\Sc$-module. The inclusion $ \Sc \rt \Sc_{t^{-1}}$ is $F_\Sc$-linear. So $W^\m(A)$ is a $F_\Sc$-finite $F_\Sc$-module. If $I$ is any ideal in $\R$ then note that
 $H^i_{I\Sc}(W^\m(A)) \cong H^i_I(W^\m(A))$ as $\R$-modules. We note that $H^i_{I\Sc}(W^\m(A)) $ are $F_\Sc$-finite $F_\Sc$-modules for $i \geq 0$. By \cite[4.1]{P2} it follows that $U_i$ and $V_i$ are
 $F_G$-finite $F_G$-modules.
 \begin{remark}
 We do not know whether $0 \rt V_{i-1} \rt H^i_I(G) \rt U_i \rt 0$ is a short exact sequence of $F_G$-modules.
 \end{remark}
\begin{remark}\label{rachel}
  Over polynomial rings $T$ shifts of graded $F_T$-finite $F_T$-module are NOT graded $F_T$-modules; see \cite[2.5(1), 4.4]{MaZhang}. Although we do not know whether this result hold for $A[\m t, t^{-1}]$ it is perhaps better if we work with $W^\m(A)$ and not with $L^\m(A) = W^\m(A)(1)$.
\end{remark}
 \s \label{char-0}Next we consider the case when $\charp k = 0$. Let $$A_d(k) = k<X_1,\ldots, X_d, \partial_1, \ldots, \partial_d>.$$
 We give the grading on $A_d(K)$ as $\deg X_i = 1$ and $\deg \partial_i = -1$. We note that $\Sc$ is also a  graded $A_d(k)$-module. We also note that $t^{-1}$ commutes with all $\partial_i$. We note that $A[t, t^{-1}]$ is also an $A_d(k)$-module and the inclusion $\Sc \rt A[t, t^{-1}]$ is $A_d(k)$-linear. So $W^\m(A)$ is also an $A_d(k)$-module.
 We note that $ 0 \rt G \rt W^\m(A)(1) \rt W^\m(A) \rt 0$ is a short exact sequence of $A_d(K)$-modules. It follows (for instance using the Cech-complex) that
 $$ \cdots \rt H^i_I(G) \rt H^i_I(W^\m(A))(1) \xrightarrow{t^{-1}} H^i_I(W^\m(A)) \rt H^{i+1}_I(G) \rt \cdots.$$
is an exact sequence of graded  $A_d(k)$-modules. Therefore
$$0 \rt V_{i-1} \rt H^i_I(G) \rt U_i \rt 0 $$
is a short exact sequence of $A_d(k)$-modules. Fix $ i \geq 0$. As $H^i_I(G)$ is a generalized Eulerian, holonomic $A_d(k)$-module it follows that $V_{i-1}$ and $U_i$ are generalized Eulerian, holonomic $A_d(k)$-modules.

 \section{Proof of Theorem \ref{m-2c}}
 In this section we give:
 \begin{proof}[Proof of Theorem \ref{m-2c}]
 Let $\wh{A}$ be the $\m$-adic completion of $A$. We note that $G_\m(A) \cong G_{\wh{\m}}(\wh{A})$ and for every $r \geq 1$ we have $G_{\m^r}(A) \cong G_{\wh{\m^r}}(\wh{A})$. So we may assume that $A$ is complete. Thus $A = k[[X_1, \ldots,X_d]]$. As $k$ is separably closed in particular $k$ is infinite.

 Let $\eta \colon A[\m^r u] \rt G_{\m^r}(A)$ be the natural map. Let $K = \eta^{-1}(J)$. We consider the exact sequence of $\Sc(\m^r) = A[\m^ru,u^{-1}]$-modules (and hence of $\R(\m^r)$-modules)
 $$ 0 \rt G_{\m^r}(A) \rt W^{\m^r}(A)(1) \xrightarrow{u^{-1}} W^{\m^r}(A) \rt 0.$$
 Taking cohomology with respect to $K$ and as $H^j_K(G_{\m^r}(A)) = 0$ for $j = d -1, d$ we have that
the map $H^{d-2}_K(W^{\m^r}(A))(1) \xrightarrow{u^{-1}} H^{d-2}_K(W^{\m^r}(A))$ is surjective, the map $H^{d-1}_K(W^{\m^r}(A))(1) \xrightarrow{u^{-1}} H^{d-1}_K(W^{\m^r}(A))$ is bijective and
 the map \\ $H^{d}_K(W^{\m^r}(A))(1) \xrightarrow{u^{-1}} H^{d}_K(W^{\m^r}(A))$ is injective.

  We make the construction as in \ref{const}.
 Let  $\xi \colon A[\m t] \rt G_\m(A)$ be the natural map. Set $I = \xi(K^\sharp)$.
 The exact sequence $(\dagger)$ induces exact sequence in cohomology
  $$\cdots  \rt H^i_I(G_\m(A)) \rt H^i_{K^\sharp}(W^\m(A))(1) \xrightarrow{t^{-1}} H^i_{K^\sharp}(W^\m(A)) \rt H^{i+1}_I(G_\m(A)) \rt \cdots. $$
  By \ref{const-prop} we have $\sqrt{(K^\sharp)^{<r>}} = \sqrt{K}$.
  Notice that multiplication by $t^{-r}$ is same as multiplication by $u^{-1}$.
  We have a surjective maps
$$ H^{d-2}_{K^{\sharp}}(W^\m(A))_{nr + r} \xrightarrow{t^{-r}} H^{d-2}_{K^{\sharp}}(W^\m(A))_{nr},  \quad \text{for all}  \ n \in \Z. $$
It follows that the map
$$ H^{d-2}_{K^{\sharp}}(W^\m(A))_{nr + 1} \xrightarrow{t^{-1}} H^{d-2}_{K^{\sharp}}(W^\m(A))_{nr},  \quad \text{is surjective for all}  \ n \in \Z. $$
Let $V_{d -2} = \coker(H^{d-2}_{K^{\sharp}}(W^\m(A))(1) \xrightarrow{t^{-1}} H^{d-2}_{K^{\sharp}}(W^\m(A)))$. Then we have \\ $(V_{d-2})_{nr} = 0$ for all $n \in \Z$.

Let $V_{d -1} = \coker(H^{d-1}_{K^{\sharp}}(W^\m(A))(1) \xrightarrow{t^{-1}} H^{d-1}_{K^{\sharp}}(W^\m(A)))$. Then by a similar argument  we have  $(V_{d-2})_{nr} = 0$ for all $n \in \Z$.

We have injective  maps
$$ H^{d-1}_{K^{\sharp}}(W^\m(A))_{nr + r} \xrightarrow{t^{-r}} H^{d-1}_{K^{\sharp}}(W^\m(A))_{nr},  \quad \text{for all}  \ n \in \Z. $$
It follows that the map
$$ H^{d-2}_{K^{\sharp}}(W^\m(A))_{nr + r} \xrightarrow{t^{-1}} H^{d-2}_{K^{\sharp}}(W^\m(A))_{nr + r -1},  \quad \text{is injective for all}  \ n \in \Z. $$
Let $U_{d -1} = \ker(H^{d-1}_{K^{\sharp}}(W^\m(A))(1) \xrightarrow{t^{-1}} H^{d-1}_{K^{\sharp}}(W^\m(A)))$. Then we have \\ $(U_{d-1})_{nr + r -1} = 0$ for all $n \in \Z$.

Let $U_{d} = \ker(H^{d}_{K^{\sharp}}(W^\m(A))(1) \xrightarrow{t^{-1}} H^{d}_{K^{\sharp}}(W^\m(A)))$. Then by a similar argument  we have  $(U_{d})_{nr + r -1} = 0$ for all $n \in \Z$.

Set $G = G_\m(A)$. If $\charp k = p$ then $V_{d-2}, V_{d-1}, U_{d-1}, U_{d-2}$ are graded $F_G$-finite $F_G$-modules. Each of these modules vanish for infinitely many $n \gg 0$ and infinitely many $n \ll 0$. So by \ref{cor-inf-vanish} each of them are zero.

If $\charp k = 0$ then $V_{d-2}, V_{d-1}, U_{d-1}, U_{d-2}$ are graded, generalized Eulerian, holonomic $A_d(k)$-modules. Each of these modules vanish for infinitely many $n \gg 0$ and infinitely many $n \ll 0$. So by \ref{cor-inf-vanish} each of them are zero.

As we have short exact sequences $ 0 \rt V_{d-2} \rt H^{d-1}_I(G) \rt U_{d-1} \rt 0$ and $0 \rt V_{d-1} \rt H^{d}_I(G) \rt U_d \rt 0$ it follows that $H^{j}_I(G) = 0$ for $j \geq d -1$.
It follows that $\dim G/I \geq 2$ and $\Proj G/I$ is connected. As $G/I = A[\m t]/K^\sharp$. It follows that $\dim A[\m^r t]/(K^\sharp)^{<r>} \geq 2$ and $\Proj A[\m^r t]/(K^\sharp)^{<r>}$ is connected. By \ref{const-prop}, we get  that $\dim G_{\m^r}(A)/J \geq 2$ and $\Proj G_{\m^r}(A)/J$ is connected. The result follows.

 \end{proof}

\end{document}